\documentclass[a4paper, 10pt, twoside, notitlepage, reqno]{amsart}

\usepackage[utf8]{inputenc}
\usepackage{color}
\usepackage{amsmath}
\usepackage{amssymb}
\usepackage{amsthm}
\usepackage{graphicx}
\usepackage{esint}
\usepackage[colorlinks=true,linkcolor=blue]{hyperref}
\usepackage{verbatim}
\usepackage{geometry}

\theoremstyle{plain}
\newtheorem{thm}{Theorem}
\newtheorem{prop}{Proposition}[section]
\newtheorem{lem}[prop]{Lemma}

\newtheorem{defi}[prop]{Definition}
\newtheorem{rmk}[prop]{Remark}

\newcommand {\R} {\mathbb{R}} 
 \newcommand {\N} {\mathbb{N}}
\newcommand {\C} {\mathbb{C}} 

\newcommand {\p} {\partial}

\newcommand {\D} {\Delta}

\newcommand {\supp} {\text{supp}}

\DeclareMathOperator{\di}{div}

\DeclareMathOperator {\Imm} {Im}

\DeclareMathOperator{\Res}{Res}

\usepackage{xspace}

\begin{document}
\title[Higher Regularity for the Signorini Problem]{Higher Regularity for the Signorini Problem for the Homogeneous, Isotropic Lam\'e System}

\author{Angkana Rüland}
\address{Ruprecht-Karls-Universität Heidelberg, Institut für Angewandte Mathematik, Im Neuenheimer Feld 205, 69121 Heidelberg, Germany}
\email{Angkana.Rueland@uni-heidelberg.de}

\author{Wenhui Shi}
\address{9 Rainforest Walk, Level 4, Monash University, Clayton 3168, VIC, Australia}
\email{wenhui.shi@monash.edu}

\begin{abstract}
In this note we discuss the (higher) regularity properties of the Signorini problem for the homogeneous, isotropic Lam\'e system. Relying on an observation by Schumann \cite{Schumann1}, we reduce the question of the solution's and the free boundary regularity for the homogeneous, isotropic Lam\'e system to the corresponding regularity properties of the obstacle problem for the half-Laplacian.
\end{abstract}

\maketitle

\section{The Lam\'e Problem for Homogeneous, Isotropic Materials}

\subsection{Set-up} The Signorini problem consists in finding the equilibrium position of an elastic body resting on a rigid surface \cite{S33,S59}. When the body is isotropic, homogeneous and when the surface is flat, the problem can be formulated locally as finding local minimizers $u=(u^1,\cdots, u^n)$, $n\geq 2$, to the functional
\begin{align}
\label{eq:min1}
J(u)=\int_{B_1^+} \frac{\mu}{2} |\nabla u+(\nabla u)^T|^2 + \lambda (\di u)^2 + F  u \ dx
\end{align}
over the closed convex set
\begin{align}
\label{eq:min2}
\mathcal{K}_0:=\{u=(u^1,\cdots, u^n)\in W^{1,2}(B_1^+; \R^n): u^n\geq 0 \text{ on } B'_1 \},
\end{align}
i.e. $J(u)\leq J(v)$ for all $v\in \mathcal{K}_0$ with $v-u=0$ on $\p B_1\cap \{x_n>0\}$. Here $\lambda, \mu>0$ are the Lam\'e constants, $F\in L^p(B_1^+;\R^n)$ with $p>2n$ is a given inhomogeneity, $Fu := \sum\limits_{j=1}^{n} F^j u^j$, $B_1=\{x\in \R^n:|x|<1\}$, $B_1^+=B_1\cap \{x_n>0\}$ and $B'_1=B_1\cap \{x_n=0\}$. 
It was shown in \cite{Schumann1} that local minimizers are $C^{1,\alpha}_{loc}(B_1^+\cup B'_1)$ regular for some $\alpha\in (0,1)$. Furthermore, they satisfy the Euler-Lagrange equations
\begin{align}\label{eq:Lame_EL}
\begin{split}
\mu\Delta u^j+ (\mu+\lambda) \p_j(\di u) &=F^j \text{ in } B_1^+, \ j\in \{1,\cdots, n\},\\
\p_nu^j + \p_j u^n &=0 \text{ on } B'_1, \ \ j\in \{1,\cdots, n-1\},
\end{split}
\end{align}
and the Signorini boundary condition
\begin{align}\label{eq:Lame_EL2}
u^n\geq 0, \quad 2\mu\p_n u^n + \lambda \di u \leq 0, \quad u^n (2\mu\p_n u^n + \lambda \di u)=0 \text{ on } B'_1.
\end{align}
The set $\Lambda_u:=\{x\in B'_1: u^n(x)=0\}$ is called the contact set and $\Gamma_u:=\p\{x\in B'_1:u^n(x)>0\}\cap B'_1$ is the free boundary. 

In this article we are interested in the optimal regularity of the solution and the so-called regular and singular free boundaries, cf.~Definition \ref{def:reg_fbp} and \eqref{eq:sing_bdry} for the precise definitions. Before stating our results, we give a brief overview of the literature. There has been an extensive study of  the Signorini problem for scalar equations, for a (non-exhaustive) overview we refer to \cite{AC04,ACS08,GP09, KPS, DSS16, FS16,CSV20} for the Laplace operator, to \cite{Gu, GSVG14, GPSVG15, KRS14, KRSI, RS17} for variable coefficient second order elliptic operators as well as to \cite{ROS17,RD20} for fully nonlinear elliptic operators. For further background and references we point to the survey articles \cite{DS17, FR20} and the book \cite{PSU}. It is worth mentioning that under a decay assumption on the solution, the Signorini problem for Laplacian in the upper half space $\R^{n}_+:=\R^n\cap \{x_n>0\}$ is equivalent to the obstacle problem for the half-Laplacian $(-\Delta)^{\frac{1}{2}}$ on $\R^{n-1}\times \{0\}$ via the Dirichlet-to-Neumann mapping $u\mapsto -\p_n \bar u\big|_{\R^{n-1}\times \{0\}}$, where $\bar u$ is the solution to the Poisson equation $\Delta \bar u=0$ in $\R^n_+$ with the Dirichlet boundary value $\bar u=u$ on $\R^{n-1}\times \{0\}$, cf. \cite{CSS,CS07}. The obstacle problem for the fractional Laplacian $(-\Delta)^{s}$ with $s\in (0,1)$ and more general integro-differetial operators were studied in \cite{S07,CSS, CRS17, AR20, JN17, KRS16}.

Compared with the scalar case, there are fewer results concerning the regularity properties for the original Signorini problem of elasticity. One of the main reasons is that the common techniques for scalar equations such as comparison principles and monotonicity formulas, in general, do not apply to non-diagonal systems. Existence of weak solutions for general non-isotropic, inhomogeneous bodies was shown by Fichera \cite{Fi63} via variational inequalities (cf. also \cite{DL76}). Concerning the Signorini problem for the Lam\'e system \eqref{eq:Lame_EL}--\eqref{eq:Lame_EL2}, Schumann proved the $C^{1,\alpha}_{loc}$ regularity of the solution by transforming the problem into a scalar obstacle problem for a pseudo-differential operator on the boundary \cite{Schumann1}. In \cite{And13} Andersson showed that when $F=0$ the optimal regularity of the solution is $C^{1,\frac{1}{2}}_{loc}$ and the regular free boundary is $C^{1,\gamma}$ for some $\gamma\in (0,1)$. Instead of reducing the vectorial problem to a scalar problem, Andersson developed a linearization technique which is directly applicable to the vectorial problem and which, in particular, does not rely on any comparison principle.

\subsection{Main results} Following the approach of Schumann \cite{Schumann1}, we will use the ``Dirichlet-to-Neumann map'' of a so-called boundary contact problem, cf.~\eqref{eq:Lame}, to reduce the vectorial problem to a scalar problem on the boundary. However, working with constant coefficient operators and seeking to derive \emph{higher} regularity results, different from \cite{Schumann1}, we locally obtain the \emph{full symbol} of the ``Dirichlet-to-Neumann'' map (instead of the principal symbol only). This is achieved by first extending the problem to the whole upper half space, then computing the explicit solution to an ODE system originating from the Fourier transform of \eqref{eq:Lame} in tangential directions and finally relating this to a corresponding obstacle problem for the half-Laplacian. In particular, it turns out that the \emph{full} symbol of this ``Dirichlet-to-Neumann map'' is, up to a constant, the half-Laplacian (see Propositions \ref{prop:Lame_ext} and \ref{prop:Lame_to_obst}). Therefore, by invoking the known regularity results for the obstacle problem for the half-Laplacian  we first obtain the optimal regularity of the solution to the Signorini problem.

\begin{thm}\label{thm:opt_reg}
Let $u \in H^1(B_1^+;\R^n)$ be a solution to \eqref{eq:Lame_EL}--\eqref{eq:Lame_EL2} with $F\in L^p(B_1^+;\R^n)$, $p>2n$. Then $u\in C^{1,\frac{1}{2}}_{loc}(B_{1}^+\cup B'_{1})$.  
\end{thm}

Building on this, our next result concerns the higher regularity of the \emph{regular} free boundary with a rather general obstacle $\varphi$ and inhomogeneity $F$. A point $x_0\in \Gamma_u$ is \emph{regular} if and only if, after subtracting an affine solution to the Signorini problem, the vanishing order of $u-\varphi$ at $x_0$ is (strictly) less than $2$, cf. Lemma \ref{lem:reg_point}.

\begin{thm}\label{thm:higher_reg}
Let $u \in H^1(B_1^+;\R^n)$ be a solution to the thin obstacle problem for the Lam\'e system \eqref{eq:Lame1} with obstacle $\varphi$ and inhomogeneity $F$. 
\begin{itemize}
\item[(i)] Suppose that $\varphi\in C^{\theta+\frac{1}{2}}(B'_1)$ and $F\in C^{\theta-\frac{3}{2}}(B_1^+\cup  B'_1;\R^n)$ with $\theta>2$ and $\theta \notin \N$, $\theta+\frac{1}{2}\notin \N$. Let $x_0\in B_{1}'$ be a regular free boundary point. Then the free boundary is $C^\theta$ in a neighbourhood of $x_0$.
\item[(ii)] Suppose that $\varphi\in C^{\omega}(B'_1)$ and $F\in C^{\omega}(B_1^+\cup B'_1;\R^n)$. Then for any regular free boundary point $x_0 \in B_{1}'$ the free boundary is $C^{\omega}$ in a neighbourhood of $x_0$.
\end{itemize}
\end{thm}

We remark that while Theorem \ref{thm:opt_reg} for $F=0$ was already proved in \cite{And13}, the higher regularity result of Theorem \ref{thm:higher_reg} is completely new, even for $F=0$. It relies on an analogous result for the thin obstacle problem for the half Laplacian, see \cite{AR20} and \cite{KRSIV}.

Similarly as in our study of the regular free boundary, also a one-to-one correspondence between the singular free boundary can be established for the Sigorini problems for the Lam\'e operator and the Laplacian. Here, for $u$ a solution to \eqref{eq:min1}-\eqref{eq:min2} we say that $x_0 \in \Sigma(u)\cap B_1'$ is an element of the \emph{singular free boundary} for the Signorini problem for the Lam\'e operator if $x_0 \in \Gamma_u\cap B_1'$ and
\begin{align}
\label{eq:sing_bdry}
\lim\limits_{r \rightarrow 0} \frac{\mathcal{H}^{n-1}(\Lambda_u \cap B_{r}'(x_0))}{\mathcal{H}^{n-1}(B_r'(x_0))} =0, \quad B_r'(x_0)\subset B_1'.
\end{align}

Working with this definition, the structure results for the Signorini problem for the Laplacian then directly translate into corresponding results for the Lam\'e operator. As a sample result, using the results from \cite{CSV20}, we for instance obtain the following stratification of the singular set:

\begin{thm}
\label{thm:Lame_sing_set}
Let $u \in H^1(B_1^+; \R^n)$ be a solution to the thin obstacle problem for the Lam\'e system \eqref{eq:min1}--\eqref{eq:min2}. Then,
\begin{align*}
\Sigma(u) = \bigcup\limits_{d=0}^{n-2} \Sigma_d(u),
\end{align*}
where the sets $\Sigma_d(u)\subset B_1'$ are contained in a countable union of $d$-dimensional $C^{1,log}$ manifolds. 
\end{thm}

We remark that lacking direct monotonicity formulae for the Signorini problem for the Lam\'e operator, the \emph{singular free boundary} had not been studied earlier.

We expect that the equivalence between the Laplacian and the full Lam\'e Signorini problems can be exploited further, thus allowing to transfer more results from the scalar to this vectorial problem.

\subsection{Outline} The remainder of the article is organized as follows: In Section \ref{sec:D-t-N} we introduce the auxiliary boundary contact problem and compute its symbol in the half-space setting. Exploiting this, in Section \ref{sec:reduc} we show that free boundary regularity for the isotropic, inhomogeneous Lam\'e problem can be reduced to that of the half-Laplacian. Finally, in Section \ref{sec:higher_reg} we invoke the regularity results for the half-Laplacian to prove the analogous statements (Theorems \ref{thm:opt_reg}-\ref{thm:Lame_sing_set}) for the Lam\'e problem.

\subsection{Notation}
In the following sections we will mainly use rather standard notation. We however point out the following conventions:
\begin{itemize}
\item For $\R^n_+:=\{x\in \R^n: \ x_n>0\}$ we set $B_1^+(x_0):=\{x\in \R^n_+ : \ |x- x_0|<1 \}$; for $x_0 \in \R^{n-1}\times \{0\}$, $B_1'(x_0):=\{x\in \R^{n-1}\times \{0\}: \ |x-x_0|<1\}$. For convenience of notation, we also identify $\R^{n-1}\times \{0\}$ with $\R^{n-1}$ and write $B_1^+:= B_1^+(0)$ as well as $B_1':=B_1'(0)$.
\item We use both the notation $C^{k,\alpha}$ and $C^{k+\alpha}$ for the Hölder spaces with $k\in \N_0$ and $\alpha \in (0,1]$. With slight abuse of notation, in our reference to function spaces, we mostly suppress the image spaces. For instance we use the notation $H^1(B_1^+)$ both for vector and scalar valued function of the corresponding regularity if there is no confusion possible.
\item The Fourier transform of a function $u:\R^n \rightarrow \R$ is denoted by $\hat{u}$. With slight abuse of notation, if there is no possibility of misunderstanding, we do not distinguish between the tangential and full, tangential and normal Fourier transforms in our notation. 
\item We recall that a \emph{weak solution} $u:B_1^+ \rightarrow \R^n$ to \eqref{eq:min1}--\eqref{eq:min2} refers to a solution of the corresponding variational inequality \cite{DL76, Schumann1}, i.e. $u\in H^1(B_1^+; \R^n)$ is such that for all $\zeta \in \mathcal{K}_0$ with $\zeta-u=0$ on $\p B_1\cap \{x_n>0\}$ it holds
\begin{align*}
a(u,\zeta- u) \geq (F,\zeta- u)_{L^2(B_1^+)},
\end{align*}
where for $u,v \in H^1(B_1^+;\R^n )$
\begin{align*}
a(u,v)&:= \int\limits_{B_1^+} 4\mu e(u): e(v) + \lambda \di(u) \di(v) dx , \\
e(u)&:= \frac{1}{2}(\nabla u + (\nabla u)^T), \ (A:B)_{ij} =A_{ij} B_{ij} \mbox{ for } A, B\in \R^{n\times n},
\end{align*} 
and $(u,v)_{L^2(B_1^+)}:= \sum\limits_{j=1}^{n} \int\limits_{B_1^+}  u^j v^j dx$. 
\item Similarly, a weak solution to obstacle problem for the half-Laplacian in $\R^{n-1}$ with obstacle $\psi \in L^2(\R^{n-1})$ is defined as a function $\tilde{v} \in H^{\frac{1}{2}}(\R^n)$ such that for all $\tilde{\zeta}\in H^{\frac{1}{2}}(\R^{n-1})$ with $\tilde{\zeta}\geq \psi$ a.e. in $\R^{n-1}$ it holds
\begin{align*}
\tilde{a}(\tilde{v}, \tilde{\zeta}-\tilde{v}) \geq 0,
\end{align*}
where $\tilde{a}(\tilde{u}_1,\tilde{u}_2) := \int\limits_{\R^{n-1}}(-\D)^{\frac{1}{4}} \tilde{u}_1 (-\D)^{\frac{1}{4}} \tilde{u}_2 dx$.
\end{itemize}

\section{Computation of the Auxiliary Operator}\label{sec:D-t-N}

We consider the following boundary contact problem 
\begin{align}
\label{eq:Lame}
\begin{split}
\mu \D u^j + (\mu + \lambda) \p_j (\di u) & = 0 \mbox{ in } \R^{n}_+,\ j \in\{1,\dots,n\},\\
\p_n u^j +  \p_j u^n & = 0 \mbox{ on } \R^{n-1} \times \{0\}, \ j \in \{1,\dots,n-1\},\\
u^n & = \varphi \mbox{ on } \R^{n-1} \times \{0\},\\
u(x)&\rightarrow 0 \mbox{ as } |x|\rightarrow \infty.
\end{split}
\end{align}
Following Schumann \cite{Schumann2}, we study an associated $\Psi$DO given by
\begin{align}\label{eq:DtN}
P \varphi := (-2 \mu \p_n u^n - \lambda \di u)|_{\R^{n-1} \times \{0\}}.
\end{align}
This problem should be viewed in analogy to the Caffarelli-Silvestre extension \cite{CS07}. It is a constant coefficient problem, so the full symbol of the operator $P$ can be computed by reducing the problem to a system of second order ODEs after Fourier transforming in the tangential directions. We present the details of this in the remainder of this section, following the computations in \cite{Schumann2}. 

We denote the bulk operator $\mu \D + (\mu + \lambda)\nabla \di  $ by $a(D', D_t)$, where $a(\xi)=\mu|\xi|^2 I_{n}+(\mu+\lambda)\xi\otimes\xi$ for $\xi\in \R^n$, $D_j=\frac{1}{i}\p_j$, $D'=\frac{1}{i}\nabla'$ and $t:=x_n$. Applying the Fourier transform in the tangential directions to \eqref{eq:Lame}, one obtains the following ODE system
\begin{align}\label{eq:ODE_Lame}
\begin{split}
a (\xi', D_t) \hat u(\xi',t)=0& \text{ for } t>0, \\
D_t \hat u^j (\xi',0)+ \xi_j \hat u^n(\xi', 0)=0,\quad \hat u^n(\xi',0)=\hat \varphi(\xi')& \text{ for } t=0.
\end{split}
\end{align}
We first compute the fundamental solution $W=(W_\ell^j)_{n\times n}$ of the ODE system 
\begin{align}
a(\xi', D_t) W_{\ell}(\xi', t)=0 \text{ for } t>0, \quad \xi'\neq 0
\end{align}
with prescribed Dirichlet data on $t=0$ and decay at $t\rightarrow \infty$, following the ideas of Schumann \cite{Schumann2}. Here $W_\ell=(W_\ell^1,\cdots, W_\ell^n)^T$ for $\ell\in \{1,\cdots, n\}$. 

To this end, we begin by considering the whole space fundamental solution to the (whole space) Lam\'e system. In complete (i.e. in tangential and normal) Fourier variables it is given by
\begin{align*}
\hat{F}^j_{\ell}(\xi) = \frac{\delta_{j \ell}}{\mu |\xi|^2} - \frac{(\lambda + \mu)}{(2\mu + \lambda)\mu} \frac{\xi_j \xi_{\ell}}{|\xi|^4}, \ j, \ell \in \{1,\dots,n\},
\end{align*}
cf. \cite[Proposition 10.14]{Mitrea13}.
We seek to transform this back in the normal direction $t$ with $t>0$. Using the residue theorem, we obtain that the corresponding solution for $t>0$ and $\xi'\neq 0$ is given by
\begin{align}
\label{eq:int_res}
\begin{split}
W^j_{\ell}(\xi', t) &= \int\limits_{\R} e^{it \xi_n} \hat{F}_{\ell}^j(\xi',\xi_n)\ d\xi_n =  \int\limits_{\Gamma^+(\xi')} e^{i t \xi_n} \hat F^{j}_{\ell}(\xi', \xi_n)\ d\xi_n \\
&= 2\pi i \Res\limits_{\xi_n=i|\xi'|} \left(e^{it\xi_n} \hat{F}^{j}_{\ell}(\xi',\xi_n)\right).
\end{split}
\end{align}
Here $\Gamma^+(\xi')$ is a positively oriented contour in the complex upper half-plane $\{\xi_n\in \mathbb{C}: \Imm \xi_n>0\}$ which encloses the point $\xi_n= i|\xi'|$ in which the complex continuation of $\hat{F}^j_{\ell}(\xi)$ ceases to be holomorphic as a function of $\xi_n \in \C$. Indeed, this contour is obtained as a deformation of the real axis by using the fact that due to $t>0$ and the decay of $\hat{F}^j_{\ell}(\xi)$ in $|\xi|\rightarrow \infty$ the integral over the boundary of an upper half-ball with increasing radius vanishes as the radius tends to infinity.

Invoking the residue theorem to compute this integral \eqref{eq:int_res}, i.e.
\begin{align*}
\frac{1}{2\pi i}W^{j}_\ell (\xi',t)&=\frac{1}{\mu}\Res\limits_{\xi_n=i|\xi'|}\left[ \frac{\delta_{j\ell}}{|\xi|^2} e^{it\xi_n} \right]-\frac{\lambda+\mu}{(2\mu+\lambda)\mu}\Res\limits_{\xi_n=i|\xi'|} \left[\frac{\xi_j\xi_\ell}{|\xi|^4 }e^{it\xi_n}\right]\\
&=\frac{1}{\mu}\lim_{\xi_n\rightarrow i|\xi'|}(\xi_n-i|\xi'|)\frac{\delta_{j\ell}}{|\xi|^2} e^{it\xi_n}  -\frac{\lambda+\mu}{(2\mu+\lambda)\mu}\lim_{\xi_n\rightarrow i|\xi'|} \frac{d}{d\xi_n}\left[(\xi_n-i|\xi'|)^2 \frac{\xi_j\xi_\ell}{|\xi|^4 }e^{it\xi_n}\right],
\end{align*}
we obtain
\begin{align}\label{eq:fundamental}
W(\xi', t) = 2\pi e^{-|\xi'| t} 
\begin{pmatrix} \frac{1}{2\mu |\xi'|}- \frac{\kappa \xi_1^2}{|\xi'|^3}- \frac{\kappa \xi_1^2 t}{|\xi'|^2} & \frac{-\kappa \xi_1 \xi_2}{|\xi'|^3} - \frac{\kappa \xi_1 \xi_2 t}{|\xi'|^2} &\cdots& \frac{-i \kappa \xi_1 t}{|\xi'|}\\
\frac{-\kappa \xi_1 \xi_2}{|\xi'|^3 } - \frac{\kappa \xi_1 \xi_2 t}{|\xi'|^2} &  \frac{1}{2\mu |\xi'|}- \frac{\kappa \xi_2^2}{|\xi'|^3}- \frac{\kappa \xi_2^2 t}{|\xi'|^2} & \cdots & \frac{-i \kappa \xi_2 t}{|\xi'|}\\
\cdots & \cdots & \cdots & \cdots\\
\frac{-i \kappa \xi_1 t}{|\xi'|} &\frac{-i \kappa \xi_{2} t}{|\xi'|} & \cdots& \frac{\nu}{|\xi'|} + \kappa t 
  \end{pmatrix},
\end{align}
where $\kappa=\frac{\lambda+\mu}{4\mu(2\mu+\lambda)}$ and $\nu=\frac{1}{2\mu}-\kappa$. 

We claim that there exist constants $C_1(\xi'),\cdots, C_n(\xi')$ such that \eqref{eq:ODE_Lame} is satisfied for $\hat u(\xi',t):=W(\xi',t)(C_1,\cdots, C_n)^T$. To see that this is the case, we first compute
\begin{align*}
\hat u(\xi', 0) = \frac{2\pi}{|\xi'|} 
\begin{pmatrix} \frac{1}{2\mu }- \frac{\kappa \xi_1^2}{|\xi'|^2} & \frac{-\kappa \xi_1 \xi_2}{|\xi'|^2} &\cdots& 0\\
\frac{-\kappa \xi_1 \xi_2}{|\xi'|^2 } &  \frac{1}{2\mu }- \frac{\kappa \xi_2^2}{|\xi'|^2} & \cdots &0\\
\vdots & \vdots & \vdots & \vdots\\
0&0 & \cdots&  \nu
\end{pmatrix}
\begin{pmatrix}
C_1(\xi') \\ C_2(\xi') \\ \vdots \\ C_n(\xi')
\end{pmatrix},
\end{align*}
and
\begin{align*}
\p_t \hat u(\xi', 0) 
&= - 2\pi
\begin{pmatrix} \frac{1}{2\mu }- \frac{\kappa \xi_1^2}{|\xi'|^2} & \frac{-\kappa \xi_1 \xi_2}{|\xi'|^2} &\cdots& 0\\
\frac{-\kappa \xi_1 \xi_2}{|\xi'|^2 } &  \frac{1}{2\mu }- \frac{\kappa \xi_2^2}{|\xi'|^2} & \cdots &0\\
\vdots & \vdots & \vdots & \vdots\\
0&0 & \cdots&  \nu
\end{pmatrix}
\begin{pmatrix}
C_1(\xi') \\ C_2(\xi') \\ \vdots \\ C_n(\xi')
\end{pmatrix}\\
& \quad + 2\pi
\begin{pmatrix}
\frac{-\kappa \xi_1^2}{|\xi'|^2}
 & \frac{-\kappa \xi_1 \xi_2}{|\xi'|^2} &\cdots& \frac{-i \kappa \xi_1}{|\xi'|}\\
\frac{-\kappa \xi_1 \xi_2}{|\xi'|^2 } &  - \frac{\kappa \xi_2^2}{|\xi'|^2} & \cdots &\frac{-i \kappa \xi_2}{|\xi'|}\\
\vdots & \vdots & \vdots & \vdots\\
\frac{-i \kappa \xi_1}{|\xi'|}& \frac{-i \kappa \xi_2}{|\xi'|} & \cdots&  \kappa
\end{pmatrix}
\begin{pmatrix}
C_1(\xi') \\ C_2(\xi') \\ \vdots \\ C_n(\xi')
\end{pmatrix}\\
& = - 2\pi
\begin{pmatrix} \frac{1}{2\mu } &  0 &\cdots& \frac{i \kappa \xi_1}{|\xi'|} \\
0&  \frac{1}{2\mu } & \cdots & \frac{i \kappa \xi_2}{|\xi'|} \\
\vdots & \vdots & \vdots & \vdots\\
\frac{i \kappa \xi_1}{|\xi'|}&\frac{i \kappa \xi_2}{|\xi'|} & \cdots&  \nu - \kappa
\end{pmatrix}
\begin{pmatrix}
C_1(\xi') \\ C_2(\xi') \\ \vdots \\ C_n(\xi')
\end{pmatrix}.
\end{align*}
From the equation $u^n = \varphi$ on $\R^{n-1} \times \{0\}$ we thus infer that
\begin{align*}
C_n(\xi') = \frac{|\xi'|}{2\pi \nu} \hat{\varphi}(\xi').
\end{align*}
The equations
\begin{align*}
\p_n u^j + \p_j u^n = 0 \mbox{ on } \R^{n-1}\times \{0\}
\end{align*}
turn into 
\begin{align*}
- \frac{1}{i}\left[ \frac{1}{2\mu} C_j(\xi') + \frac{i \kappa \xi_j}{|\xi'|} C_n(\xi') \right] + \frac{\nu \xi_j}{|\xi'|}C_n(\xi') = 0.
\end{align*}
This yields
\begin{align*}
C_j(\xi') = \frac{2\mu (\kappa - \nu)\xi_j}{i 
|\xi'|} C_n(\xi')
=  \frac{2 \mu (\kappa - \nu)\xi_j}{ 2\pi i \nu} \hat{\varphi}(\xi') \mbox{ for } j \in \{1,\dots,n-1\}.
\end{align*}
Using these observations, we seek to compute
\begin{align*}
P \varphi = -2 \mu \p_n u^n - \lambda \text{div}(u).
\end{align*}
To this end, we first note that for $j \in \{1,\dots,n-1\}$
\begin{align*}
\hat{u}^j(\xi',0)
&= \frac{2\pi }{ |\xi'|} \left( \frac{1}{2 \mu} C_j(\xi') - \frac{\kappa}{|\xi'|^2}\sum\limits_{\ell=1}^{n-1}\xi_j \xi_{\ell} C_{\ell}(\xi') \right)\\
&= \frac{2\pi }{ |\xi'|} \left( \frac{1}{2 \mu} C_j(\xi') - \frac{\kappa (\kappa - \nu)}{|\xi'|^2} \xi_j \frac{2 \mu}{2\pi i\nu} \hat{\varphi}(\xi)\sum\limits_{\ell=1}^{n-1}\xi_{\ell}^2 \right)\\
& = \frac{2\pi }{ |\xi'|} \left( \frac{1}{2\mu} C_j(\xi') - \kappa (\kappa-\nu) \xi_j \frac{2 \mu}{2\pi i \nu} \hat{\varphi}(\xi') \right)\\
& = \frac{2\pi }{|\xi'|}\left( \frac{(\kappa - \nu)}{2\pi i \nu} \xi_j - \frac{2 \mu \kappa (\kappa - \nu)}{2\pi i \nu} \xi_j \right) \hat{\varphi}(\xi')\\
& = \frac{\xi_j}{i|\xi'|}  \frac{(\kappa - \nu)}{\nu}(1-2\mu \nu) \hat{\varphi}(\xi').
\end{align*}
As a consequence,
\begin{align}
\label{eq:tangential_div}
\begin{split}
\sum\limits_{j=1}^{n-1} i \xi_j \hat{u}^j(\xi',0) = \sum\limits_{j=1}^{n-1}\frac{\xi_j^2}{|\xi'|}\frac{(\kappa - \nu)}{\nu}(1-2\mu \nu) \hat{\varphi}(\xi')
= |\xi'| \frac{(\kappa-\nu)}{\nu}(1-2\mu \nu) \hat{\varphi}(\xi').
\end{split}
\end{align}
Moreover,
\begin{align}
\label{eq:normal_div}
\begin{split}
\p_n \hat{u}^n(\xi',0) 
& = - 2\pi\left[ \sum\limits_{j=1}^{n-1} \frac{i \kappa \xi_j}{|\xi'|} C_j(\xi') + (\nu - \kappa) C_n(\xi') \right]\\
& = - 2\pi\left[ i \kappa \sum\limits_{j=1}^{n-1} \frac{\xi_j^2}{|\xi'|} \frac{2\mu}{2\pi i \nu} \hat{\varphi}(\xi')(\kappa - \nu) + (\nu - \kappa) \frac{ |\xi'|}{2\pi \nu} \hat{\varphi}(\xi') \right]\\
&= - \frac{(\kappa -\nu)(2\mu\kappa -1)}{\nu} |\xi'| \hat{\varphi}(\xi').
\end{split}
\end{align}
With the expressions \eqref{eq:tangential_div} and \eqref{eq:normal_div} in hand, we return to the computation of the operator $P$ for which we thus obtain
\begin{align*}
\widehat{P \varphi}(\xi') 
&= -(2\mu + \lambda) \p_n \hat{u}^n(\xi',0) - \lambda  \sum\limits_{j=1}^{n-1} i \xi_j \hat{u}^j(\xi',0)\\
&= (2\mu + \lambda) \frac{(\kappa - \nu)}{\nu}|\xi'|\hat{\varphi}(\xi')(2\mu\kappa-1) - \lambda \frac{(\kappa - \nu)}{\nu} (1-2\mu \nu) |\xi'| \hat{\varphi}(\xi')\\
&  =\frac{\kappa - \nu}{\nu} \left[ (2\mu + \lambda)(2\mu\kappa-1) - \lambda(1-2 \mu \nu) \right] |\xi'| \hat{\varphi}(\xi')\\
&=\frac{2\mu(\lambda+\mu)}{\lambda+2\mu}|\xi'|\hat{\varphi}(\xi').
\end{align*}

We summarize the discussion from above in the following Proposition.

\begin{prop}\label{prop:Lame_ext}
\begin{itemize}
\item [(i)] Given $\varphi\in H^{\frac{1}{2}}(\R^{n-1})$, there is a unique solution $u \in C^\infty_{loc}(\R^n_+)\cap \dot{H}^1(\R^n_+)$, $u(\cdot, x_n)\in H^{\frac{1}{2}}(\R^{n-1})$ for each $x_n>0$ to the boundary contact problem \eqref{eq:Lame} with $u^n(x',x_n)\rightarrow \varphi(x')$ in $H^{\frac{1}{2}}(\R^{n-1})$ as $x_{n}\rightarrow 0_+$. If additionally $\varphi\in C^{k+\alpha}_{loc}(\R^{n-1})$ for $k\geq 0$ and $\alpha\in (0,1)$, then $u\in C^{k+\alpha}_{loc}(\R^n_+\cup (\R^{n-1}\times \{0\}))$ up to the boundary.
\item [(ii)] The Dirichlet-to-Neumann map $P$ as given in \eqref{eq:DtN} is well-defined from $\dot{H}^{\frac{1}{2}}(\R^{n-1})$ to $\dot{H}^{-\frac{1}{2}}(\R^{n-1})$. The symbol of $P$ is 
\begin{align*}
\sigma(P)=c_{\lambda,\mu}|\xi'|,\quad c_{\lambda,\mu}=\frac{2\mu(\lambda+\mu)}{\lambda+2\mu}.
\end{align*}
\end{itemize}
\end{prop}

\begin{rmk}
\label{rmk:Lame_ext}
In the following, we will refer to the function $u$ from Proposition \ref{prop:Lame_ext} as the \emph{Lam\'e extension of $\varphi$}.
\end{rmk}

\begin{proof}
From the discussion above, we obtain that 
\begin{align*}
u^j(x)=\frac{1}{(2\pi)^{\frac{n-1}{2}}}\int_{\R^{n-1}} e^{i\xi'\cdot x'}W_{\ell}^j(\xi',x_n) C_{\ell}(\xi') \ d\xi',\quad j\in \{1,\cdots, n\},
\end{align*}
is a solution to \eqref{eq:Lame}.
From the expression for $W$ in \eqref{eq:fundamental} and the observation that 
\begin{align*}
C(\xi')=\left(\frac{2\mu(\kappa-\nu)\xi'}{2\pi i}\nu, \frac{|\xi'|}{2\pi \nu}\right)^T \hat \varphi(\xi') ,
\end{align*}
we deduce that
\begin{align*}
W(\xi', x_n) C(\xi')=e^{-|\xi'|x_n}
\begin{pmatrix}
\frac{-i(\lambda+\mu)\xi' x_n}{2\mu+\lambda} + \frac{i \mu \xi'}{(2\mu+\lambda)|\xi'|}\\
1+\frac{\lambda+\mu}{2\mu+\lambda}|\xi'|x_n
\end{pmatrix}
\hat \varphi(\xi').
\end{align*}
The integral converges absolutely for $x_n>0$. Moreover, by Cauchy-Schwarz and Plancherel one infers 
\begin{align*}
\|u(\cdot, x_n)\|_{H^{\frac{1}{2}}(\R^{n-1})}&\leq C\|\varphi\|_{H^{\frac{1}{2}}(\R^{n-1})}, \quad |u(x)|\leq C\|\varphi\|_{L^2(\R^{n-1})}|x_n|^{-\frac{n-1}{2}}, \quad x_n>0,\\
 \|\nabla u\|_{L^2(\R^n_+)}&\leq C \|\varphi\|_{\dot{H}^{\frac{1}{2}}(\R^{n-1})}
\end{align*}
for some $C=C(\lambda,\mu,n)$. Convergence of $u^n$ to the initial data $\varphi$ in $H^{\frac{1}{2}}(\R^{n-1})$ as $x_n\rightarrow 0_+$ is a consequence of the Plancherel identity. Note that the set of rigid displacements which satisfy the boundary conditions consists of
\begin{align*} 
 \mathcal{R}=\{v:v=a+b\wedge x,\ a, b\in \R^n;\ v^n=0 \text{ on } \R^{n-1}\times \{0\}, \ |v(x)|\rightarrow 0 \text{ as } |x|\rightarrow \infty\}=\{0\}.
\end{align*} 
Hence, by \cite[Chapter III, Theorem 3.3]{DL76} one can conclude that $u$ is the unique solution to the boundary contact problem \eqref{eq:Lame}. The up-to-the-boundary-regularity of $u$ in Hölder spaces follows from \cite{ADNII}. The remaining statements of the proposition follow from the previous computations.
\end{proof}

For later purposes we also need the up to the boundary regularity properties for solutions to the following problem
\begin{equation}\label{eq:Neumann_aux}
\begin{split}
\mu \D w^j + (\mu + \lambda) \p_j \di w & = f^j \mbox{ in } \R^n_+,\quad j\in\{1,\cdots, n\},\\
\p_n w^j + \p_j w^n & = g^j \mbox{ on } \R^{n-1} \times \{0\},\quad j\in \{1,\cdots, n-1\},\\
w^n & = 0 \mbox{ on } \R^{n-1}\times \{0\}.
\end{split}
\end{equation}
\begin{prop}
\label{prop:more_gen_prop_reg}
Given $f\in L^2(\R^n_+)$ and $g\in H^{\frac{1}{2}}(\R^{n-1})$, there exists a unique solution $w\in \dot{H}^2(\R^n_+)$ to \eqref{eq:Neumann_aux}. Moreover, if $f\in C^{k-1,\alpha}(B_1^+)$ and $g\in C^{k,\alpha}(B'_1)$ for $k\in \N$, $k\geq 1$, then $w\in C^{k+1,\alpha}_{loc}(B_1^+\cup B'_1)$. If $f$ is analytic in $B_1^+\cup B'_1$ and $g$ is analytic in $B'_1$, then $w$ is analytic in any (relatively) open set in $B_1^+\cup B'_1$.
\end{prop}
\begin{proof}
We first consider solutions to 
\begin{equation}\label{eq:Neumann_aux_1}
\begin{split}
\mu \D w^j + (\mu + \lambda) \p_j \di w & = 0 \mbox{ in } \R^n_+,\quad j\in\{1,\cdots, n\},\\
\p_n w^j + \p_j w^n & = g^j \mbox{ on } \R^{n-1} \times \{0\},\quad j\in \{1,\cdots, n-1\},\\
w^n & = 0 \mbox{ on } \R^{n-1}\times \{0\}.
\end{split}
\end{equation} 
By a similar argument as for Proposition \ref{prop:Lame_ext} solutions to \eqref{eq:Neumann_aux_1} 
can be expressed in the form $\hat{w}^j(\xi',t) = W(\xi',t)(C_1,\dots,C_n)^T$, with $W(\xi',t)$ as above but where now for $j\in\{1,\dots,n-1\}$ the functions $C_j(\xi') = - \frac{\mu}{\pi} \hat{g}^j(\xi')$ and  $C_n(\xi')=0$. As above this provides the unique solvability of this problem with
\begin{align*}
w^j(x)=\frac{1}{(2\pi)^{\frac{n-1}{2}}}\int_{\R^{n-1}}e^{i\xi'\cdot x'} K^j_\ell (\xi',x_n) \hat g^\ell(\xi') \ d\xi',
\end{align*}
where $K^j_\ell(\xi',x_n)=- \frac{\mu}{\pi}  W^j_\ell(\xi',x_n)$. Noting that the mixed Dirichlet-Neumann boundary condition in \eqref{eq:Neumann_aux} is elliptic in the sense of \cite{ADNII}, one further obtains Schauder estimates up to the boundary. In particular, we thus infer that if $g \in C^{k,\alpha}(B_1')$, then $w \in C^{k+1,\alpha}(B_{r}^+\cup B'_r)$ for any $r\in (0,1)$. 

Moreover, if the data are analytic near and up to the boundary, solutions are (locally) analytic. Indeed, as the system is elliptic, it suffices to show $w^j |_{B'_r}$ is real analytic and then invoke \cite{MN57}. This can be directly read off from the expression of $\hat w^j(\xi',0)$ and the characterization of analyticity through the Fourier transform (see for instance \cite[Section 8.4]{HoermanderI}). This also persists for data which are locally analytic.

Finally, we remark that solutions to
\begin{equation}\label{eq:Neumann_aux1}
\begin{split}
\mu \D w^j + (\mu + \lambda) \p_j \di w & = f^j \mbox{ in } \R^n_+,\quad j\in\{1,\cdots, n\},\\
\p_n w^j + \p_j w^n & = 0 \mbox{ on } \R^{n-1} \times \{0\},\quad j\in \{1,\cdots, n-1\},\\
w^n & = 0 \mbox{ on } \R^{n-1}\times \{0\},
\end{split}
\end{equation}
can be reduced to solutions of the whole space Lam\'e problem by a reflection argument (reflecting $w^n$ oddly and $w^j$ for $j\in \{1,\dots,n-1\}$ evenly). In particular, again elliptic regularity estimates can be invoked implying that if $f^j\in C^{k,\alpha}_{loc}(\R^n_+\cup (\R^{n-1}\times \{0\}))$, then $w^j \in C^{k+2,\alpha}_{loc}(\R^n_+\cup (\R^{n-1}\times \{0\}))$. Similar observations hold for the propagation of analyticity. We remark that for the problem \eqref{eq:Neumann_aux1} this can be even reduced to the propagation of analyticity for the Laplacian: After the extension of the problem to the whole space Lam\'e problem, the function $p:= \di(w)$ satisfies Laplace's equation in the whole space with data $\di(f)$. Hence $p$ is analytic. Then solving the problem for $\D w^j$ and putting the divergence contribution onto the right hand side implies the desired regularity also for the functions $w^j$, $j\in \{1,\dots,n\}$.
\end{proof}

\begin{rmk}
An alternative argument yielding analyticity of the functions $w^j|_{B_r'}$, $\p_n w|_{B_r'}$ and $\di(w)|_{B_r'}$ from \eqref{eq:Neumann_aux_1} in the \emph{tangential} directions (which suffices for our regularity discussion in Section \ref{sec:higher_reg} below) can also be obtained through an application of the analytic implicit function theorem as in \cite{KRS16, KL12}. Yet another argument proceeds by bootstrapping regularity estimates as in \cite{K96}.
\end{rmk}

\section{Reductions}
\label{sec:reduc}

In this section,
we relate the thin obstacle problem for the  Lam\'e system on $B_1^+ \subset \R^n$ and the obstacle problem for the half-Laplacian on $\R^{n-1}$.
Here, on the one hand, the (strong form of the) thin obstacle problem for the isotropic, homogeneous Lam\'e problem on the unit ball reads
\begin{align}
\label{eq:Lame1}
\begin{split}
\mu \D u^j + (\mu + \lambda) \p_j \di (u) & = F^j \mbox{ in } B_1^+, \ j\in \{1,\cdots, n\},\\
\p_n u^j + \p_j u^n & = 0 \mbox{ on } B_1',\ j\in \{1,\cdots, n-1\},\\
u^n & \geq \varphi \mbox{ on } B_1',\\
-2\mu \p_n u^n - \lambda \di(u) &\geq 0 \mbox{ on } B_1',\\
u^n (2\mu \p_n u^n + \lambda \di(u)) &= 0 \mbox{ on } B_1'.
\end{split}
\end{align}
On the other hand, for a sufficiently regular obstacle $\psi$ with suitably decaying as $|x| \rightarrow \infty$, the (strong form of the) obstacle problem for the half-Laplacian is given by 
\begin{align}
\label{eq:obstacle_half}
\min\{(-\D)^{\frac{1}{2}} u(x), u(x)-\psi(x)\} = 0 \mbox{ for } x \in \R^{n-1}.
\end{align}

We begin by showing that a solution to \eqref{eq:obstacle_half} can be recast as a solution to \eqref{eq:Lame1}:

\begin{prop}
\label{prop:obst_to_Lame}
Let $u: \R^{n-1} \rightarrow \R$ with $u\in H^{\frac{1}{2}}(\R^{n-1})$ be a solution to \eqref{eq:obstacle_half} with the obstacle $\psi \in H^{\frac{1}{2}}(\R^{n-1})$. Let $w: \R^{n}_+ \rightarrow \R^n$ with $w \in \dot{H}^1(\R^{n}_+)$ be the Lam\'e extension of $u$ as in Proposition \ref{prop:Lame_ext}. Then, the function $w$ weakly (i.e. in the sense of a variational inquality, cf. \cite{DL76}) satisfies \eqref{eq:Lame1} in $B_1^+$ with $\varphi = \psi$ and $F^j=0$.
\end{prop}

\begin{proof}
The proof directly follows from the results from the previous section. By these, the Lam\'e extension $w$ of $u$ solves
\begin{align}
\label{eq:Lame2}
\begin{split}
\mu \D  w^j + (\mu + \lambda) \p_j \di w & = 0 \mbox{ in } B_1^+,\ j\in \{1,\cdots, n\},\\
\p_n w^j + \p_j w^n & = 0 \mbox{ on } B_1',\ j\in \{1,\cdots, n-1\},\\
w^n & = u \mbox{ on } B_1'.
\end{split}
\end{align}
Moreover, 
\begin{align*}
-2\mu \p_n w^n -\lambda \di w\big|_{\R^{n-1}\times\{0\}} = c(-\D)^{\frac{1}{2}}u.
\end{align*}
for some positive constant $c=c(n,\lambda, \mu)$.
As a consequence, the function $w$ is a solution to \eqref{eq:Lame1} with $\varphi = \psi$ in $B_1^+$ as claimed.
\end{proof}

Conversely, if $u: B_1^+ \rightarrow \R^n$ with $u\in H^1(B_1^+)$ is a solution to \eqref{eq:Lame1}, then up to solving a linear problem with better regularity, it is possible to reduce the Lam\'e problem to the obstacle problem for the half Laplacian. It is this direction which we mainly employ in our analysis of the free boundary problem for the isotropic, homogeneous Lam\'e system.

\begin{prop}
\label{prop:Lame_to_obst}
Let $u\in H^1(B_1^+)$ be a solution to \eqref{eq:Lame1} with the obstacle $\varphi \in H^{\frac{1}{2}}(B_1')$ and inhomogeneity $F\in L^2(B_1^+)$.  Let $\eta: \R^n \rightarrow \R$, $\eta(x)=\eta(|x|)\geq 0$ be a smooth cut-off function with $\eta=1$ in $B_r(x_0)\subset B_1$ for some $r\in (0,1)$ and $x_0\in B'_1$, and $\eta=0$ outside $B_1$.
Then there exists a function $\bar{w} \in \dot{H}^{\frac{1}{2}}( \R^{n-1})$ such that $\tilde{v}(x'):=u^n(x',0)\eta(x',0) - \bar{w}(x') \in \dot{H}^{\frac{1}{2}}(\R^{n-1})$ is a weak solution to the obstacle problem for $(-\Delta)^{\frac{1}{2}}$ with the obstacle $\psi:=\varphi\eta- \bar{w} \in \dot{H}^{\frac{1}{2}}(\R^{n-1})$, i.e. it weakly satisfies
\begin{align*}
\min\{(-\D)^{\frac{1}{2}} \tilde{v}(x'), \tilde{v}(x')-\psi(x') \} \geq 0, \ x' \in \R^{n-1}.
\end{align*}
In particular, $\Lambda_{\tilde{v}}\big|_{\supp(\eta)}=\Lambda_u \big|_{\supp(\eta)}$, where $\Lambda_{\tilde{v}}:=\{x'\in \R^{n-1}: \tilde v(x')=\psi(x')\}$.
\end{prop}

\begin{proof}
\emph{(i) Extension.} We begin by defining $\bar u:= u\eta$, where $\eta(x)=\eta(|x|)$ is a smooth cut-off function which is equal to one in $B_r(x_0) \Subset B_1$ for some $r\in (0,1)$ and vanishes outside $B_1$.
Then, $\bar u$ solves the equations
\begin{align}\label{eq:u_cutoff}
\begin{split}
\mu \D \bar u^j + (\mu + \lambda) \p_j \di \bar u & = f^j \mbox{ in } \R^n_+, \quad j\in \{1,\cdots, n\},\\
\p_n \bar u^j + \p_j \bar u^n & = g^j\mbox{ on } \R^{n-1} \times \{0\},\quad j\in \{1,\cdots, n-1\},\\
\bar u^n &= u^n \eta \mbox{ on } \R^{n-1} \times \{0\},
\end{split}
\end{align}
where 
\begin{align*}
f^j &=  F^j\eta+2 \mu \nabla u^j \cdot \nabla \eta  + \mu u^j \D \eta + (\mu + \lambda) \sum_{k=1}^{n}\left(u^k \p_j \p_k \eta+ \p_j u^k \p_k \eta + \p_j \eta \p_k u^k\right) , \\
g^j &=  u^n \p_j \eta.
\end{align*}
Moreover,  we have that
\begin{align}\label{eq:u_cutoff_bdry}
\begin{split}
\bar u^n\geq \varphi\eta, \quad -2 \mu \p_n \bar u^n - \lambda \di \bar u &\geq  h \text{ on } \R^{n-1}\times \{0\},\\
 -2 \mu \p_n \bar u^n - \lambda \di \bar u &=h \text{ in } (\R^{n-1}\times \{0\})\cap \{\bar u^n>\varphi\eta\},
\end{split}
\end{align}
where $h= - \lambda \sum_{k=1}^{n-1}u^k \p_k \eta \big|_{\R^{n-1}\times \{0\}}$. In the above computations we have used that $\p_n \eta(x',0)=0$.
For $F^j \in L^2(B_1^+)$ and $u\in H^1(B_1^+)$ as in \eqref{eq:Lame1}, we immediately infer that $f^j \in L^2(\R^n_+)$, $g^j \in H^{\frac{1}{2}}(\R^{n-1})$ and $h\in H^{\frac{1}{2}}(\R^{n-1})$.

\emph{(ii) Auxiliary problem.} We next consider the following auxiliary problem
\begin{equation}\label{eq:Neumann}
\begin{split}
\mu \D w^j + (\mu + \lambda) \p_j \di w & = f^j \mbox{ in } \R^n_+,\quad j\in\{1,\cdots, n\},\\
\p_n w^j + \p_j w^n & = g^j \mbox{ on } \R^{n-1} \times \{0\},\quad j\in \{1,\cdots, n-1\},\\
w^n & = 0 \mbox{ on } \R^{n-1}\times \{0\}.
\end{split}
\end{equation}
with decay at $|x| \rightarrow \infty$.
By Proposition \ref{prop:more_gen_prop_reg}, this problem (with decay as $x_n \rightarrow \infty$) is uniquely solvable yielding solutions $w\in \dot{H}^2(\R^n_+)$. Further the Fourier characterization from Proposition \ref{prop:more_gen_prop_reg} also implies that $w^j \in \dot{H}^1(\R^{n-1})$ for $j\in \{1,\dots,n-1\}$ and $\p_n w^n \in L^2(\R^{n-1})$. 

We next observe that, as a consequence, the function $v:= \bar u-w \in \dot{H}^1(\R^n_+)$ satisfies
\begin{align*}
\mu \D v^j + (\mu + \lambda) \p_j \di v & = 0 \mbox{ in } \R^n_+,\quad j\in \{1,\cdots, n\}\\
\p_n v^j + \p_j v^n & = 0 \mbox{ on } \R^{n-1} \times \{0\}, \quad j\in \{1,\cdots, n-1\}\\
v^n & = u^n \eta  \geq \varphi\eta  \mbox{ on } \R^{n-1} \times \{0\},\\
-2 \mu \p_n v^n - \lambda \di v& \geq  \tilde{h} \mbox{ on } \R^{n-1} \times \{0\},\\
-2 \mu \p_n v^n - \lambda \di v& = \tilde{h} \text{ on } \{v^n>\varphi\eta\} \cap (\R^{n-1}\times\{0\}).
\end{align*}
Here $\tilde{h} = h + 2\mu \p_n w^n + \lambda \di(w) \in L^{2}(\R^{n-1})$ where we used that $w^j \in \dot{H}^1(\R^{n-1})$ for $j \in \{1,\dots,n-1\}$ and $\p_n w^n \in L^2(\R^{n-1})$.
By the results from Section \ref{sec:D-t-N} we have that $\bar{v}(x'):=v^n(x',0) $ is a weak solution to
\begin{align*}
\bar{v} \geq \varphi\eta , \ (-\D)^{\frac{1}{2}} \bar{v} \geq \tilde{h} \mbox{ in } \R^{n-1}, \
(-\D)^{\frac{1}{2}} \bar{v} = \tilde{h} \mbox{ in } \R^{n-1} \cap \{x'\in \R^{n-1}: \ \bar{v}(x')>\varphi\eta \}.
\end{align*}
Finally, denoting $\bar{w}(x') = (-\D_{x'})^{-\frac{1}{2}} \tilde{h}(x')\in \dot{H}^{1}(\R^{n-1})$ and setting $\tilde{v}(x'):= \bar v(x')-\bar{w}(x')$ and $\psi:=\varphi\eta-\bar{w}$ then concludes the argument.
\end{proof}

\begin{rmk}
\label{rmk:variations}
We remark that it is possible to carry out the above reduction by means of slight variations of the auxiliary problem \eqref{eq:Neumann} which is used in the proof of Proposition \ref{prop:Lame_to_obst}.
\end{rmk}

\section{Proofs of Theorems \ref{thm:opt_reg}-\ref{thm:Lame_sing_set}}
\label{sec:higher_reg}

In this section, we present the proofs of our main results. To this end, we use the reduction arguments from the previous sections to deduce regularity results for the Signorini problem for the Lam\'e system from corresponding ones for the obstacle problem for the half-Laplacian, after starting with some minimal, non-optimal local Hölder regularity for the Lam\'e problem. For the convenience of the reader, we first collect these facts which we will then use in the sequel to provide the arguments for Theorem \ref{thm:opt_reg}-\ref{thm:Lame_sing_set}. Recalling the fact that by virtue of the Cafferelli-Silvestre extension \cite{CS07}, when suitably localized, the obstacle problem for the half-Laplacian is equivalent to the Signorini problem for the Laplacian, we will not distinguish between references for these two problems below.

\begin{itemize}
\item[(i)] Solutions $u$ to the minimization problem \eqref{eq:min1}--\eqref{eq:min2} with $F\in L^p(B_1^+)$ for $p>2n$ are $C^{1,\alpha}_{loc}(B_1^+ \cup B_1')$ regular for some $\alpha>0$ (see \cite{Schumann1}).
\item[(ii)] Solutions $\tilde{v}$ to the obstacle problem \eqref{eq:obstacle_half} enjoy the optimal $C^{1,1/2}_{loc}(B_1')$ regularity if $\psi \in C^{1,\gamma}(B_1')$ for some $\gamma >\frac{1}{2}$ \cite[Theorem 4]{RS17} (see also \cite[Section 1.3.1]{KRS16} relating the nonlocal obstacle problem for the fractional Laplacian to an analogous local Signorini problem for a (degenerate) elliptic operator in the upper-half space). For any $r\in (0,1)$ the free boundary $\Gamma_{\tilde{v}}\cap B_{r}'$ is given by $\p \{x\in B_r': \tilde{v}(x)=\psi(x)\}$. At a free boundary point $x_0 \in \Gamma_{\tilde{v}}\cap B_{r}'$ the solution $\tilde{v}-\psi$ has a well-defined order of vanishing. For $\gamma = 1-\epsilon$ it is either equal to $3/2$ or larger than or equal to $2-\epsilon$ (see \cite[Theorem 5]{RS17} for the $C^{1,\gamma}$-regular obstacle and \cite[Chapter 9]{PSU} for the flat obstacle).
The free boundary $\Gamma_{\tilde{v}}\cap B_{r}$ separates into the \emph{regular} free boundary and its complement (see \cite[Chapter 9]{PSU}). At the regular free boundary the solution vanishes of order $3/2$. For a $C^{\theta+\frac{1}{2}}_{loc}(\R^{n-1})$ obstacle with $\theta>0$, $\theta, \theta+\frac{1}{2}\notin \N$, the regular free boundary is locally a $C^{\theta}$ regular manifold \cite{AR20}.
\item[(iii)] For a solution $\tilde{v}$ to the obstacle problem for the fractional Laplacian the singular set $\Sigma(\tilde{v})$ is defined in analogy to \eqref{eq:sing_bdry}: Indeed, $x_0 \in \Sigma(\tilde{v}) $, iff
\begin{align}
\label{eq:sing_bdry_scalar}
\lim\limits_{r \rightarrow 0} \frac{\mathcal{H}^{n-1}(\Lambda_{\tilde{v}} \cap B_{r}'(x_0))}{\mathcal{H}^{n-1}(B_r'(x_0))}=0, \ B_r'(x_0)\subset B_1'(0).
\end{align}
It is known \cite{GP09, CSV20} that if $\varphi \in C^{\infty}$, then the blow-up $p_{x_0}$ of $\tilde{v}-\varphi$ is a harmonic polynomial of order $2m$ with $m\in \N$. The singular free boundary decomposes as
\begin{align*}
\Sigma(\tilde{v}) = \bigcup\limits_{d=0}^{n-2} \Sigma_d(\tilde{v}),
\end{align*}
where 
\begin{align*}
\Sigma_d(\tilde{v}) = \{x_0 \in \Sigma(\tilde{v}): \dim \{\xi \in \R^{n-1}: \ \sum\limits_{j=1}^{n-1}\xi_j \p_{x_j} p_{x_0}(x',0) = 0 \mbox{ for all } x'\in \R^{n-1}\} = d \}.
\end{align*}
 The sets $\Sigma_d(\tilde{v})$ are contained in a coutable union of $d$-dimensional $C^{1,log}$ manifolds. 
\end{itemize}

Using the relation between the Signorini problem for the Lam\'e system and the obstacle problem for the half-Laplacian (see Proposition \ref{prop:Lame_to_obst}), as well as the observation (i) from above, we first prove Theorem \ref{thm:opt_reg}, which concerns the optimal regularity of the solution.
\begin{proof}[Proof of Theorem \ref{thm:opt_reg}]
Let $u$ be a solution to \eqref{eq:Lame_EL}--\eqref{eq:Lame_EL2} in $B_1^+$ with $F\in L^p(B_1^+;\R^n)$ and $p>2n$. Let $\bar u:=u\eta$, where $\eta$ is a smooth, radial cut-off function whose support is in $B_1$ and which is equal to one in $B_r$ for some $r\in (0,1)$. Then $\bar u$ solves \eqref{eq:u_cutoff}--\eqref{eq:u_cutoff_bdry} with compactly supported data $f^j, g^j, h$ as in the proof of Proposition \ref{prop:Lame_to_obst}. Using that by assumption $F\in L^p$ and that $u\in C^{1,\alpha}_{loc}(B_1^+\cup B'_1)$ for some $\alpha\in(0,1)$ (cf. \cite{Schumann1}), we infer that $f\in L^p(\R^n_+)$, $g\in C^{1,\alpha}(\R^{n-1})$ and $h\in C^{1,\alpha}(\R^{n-1})$. We write $\bar u=w+v$, where $w$ solves the auxiliary problem \eqref{eq:Neumann}. By the $W^{2,p}$ regularity results for the (linear)  boundary value problem \eqref{eq:Neumann} (which follows from the characterization of the operator in Proposition \ref{prop:more_gen_prop_reg} and for instance \cite{ADNII}), one has that $w\in W^{2,p}_{loc}(\R^n_+\cup (\R^{n-1}\times\{0\}))\hookrightarrow C^{1,\gamma}_{loc}(\R^{n}_+\cup (\R^{n-1}\times \{0\}))$ with $\gamma = 1-\frac{n}{p}\in (\frac{1}{2},1)$. By Proposition \ref{prop:Lame_to_obst}, $v^n\big|_{\R^{n-1}\times \{0\}}$ solves the obstacle problem for $(-\Delta)^{\frac{1}{2}}$ with the obstacle $\psi = -\overline{w} = -(-\Delta_{x'})^{-\frac{1}{2}}(-\lambda \sum_{k=1}^{n-1}u^k\p_k \eta + 2\mu \p_n w^n + \lambda \di w) \in C^{1,\gamma}_{loc}(\R^{n-1})$ (see for instance \cite[Chapter 5]{T07} for the mapping properties of $(-\D)^{- \frac{1}{2}}$ in Hölder-Zygmund spaces or \cite[Proposition 2.8]{S07}). Then we infer from the known regularity results for the thin obstacle problem with $C^{1,\gamma}$ obstacle (cf. \cite{RS17}) that $v^n(x',0)\in C^{1,\frac{1}{2}}_{loc}(\R^{n-1})$. Thus, $v\in C^{1,\frac{1}{2}}_{loc}(\R^n_+\cup (\R^{n-1}\times \{0\}))$ by the up to the boundary regularity for the boundary contact problem \eqref{eq:Lame}, cf. Proposition \ref{prop:Lame_ext}. This together with the regularity of $w$ implies that $u$ is $C^{1,\frac{1}{2}}_{loc}$ up to the boundary. 
\end{proof}

We next approach the regularity of the regular free boundary for the Signorini problem for the Lam\'e system. Using the Dirichlet-to-Neumann map associated to the boundary contact problem \eqref{eq:Lame}, we reduce the results to the analogous ones for the half-Laplacian.

We begin by recalling the notion of a regular point for the Signorini problem for the Lam\'e system. 

\begin{defi}\label{def:reg_fbp}
Let $e_n = (0,\dots,0,1)\in \R^n$ denote the $n$-th unit vector.
Let $u\in H^1(B_1^+)$ be a solution to \eqref{eq:Lame_EL}--\eqref{eq:Lame_EL2}.
A point $x_0 \in \Gamma_{u}$ is called a \emph{regular free boundary point} if for an affine solution $p_1$ of the Lam\'e system and for some $\xi \in \R^{n-1} \setminus \{0\}$
\begin{align*}
\lim\limits_{r \rightarrow 0} \frac{(u-p_1)(x_0+rx)}{r^{\frac{3}{2}}} = p_{\frac{3}{2}}^{\xi}(x).
\end{align*}
Here $p_{\frac{3}{2}}^{\xi}(x) = |\xi|p_{\frac{3}{2}}(\frac{\xi}{|\xi|}\cdot x', x_n)$, where for $\hat{e}\in \mathbb{S}^{n-1}\cap\{x_n=0\}$,  $p_{\frac{3}{2}}(\hat{e}\cdot x',x_n):\R^n_+\rightarrow \R^n$ is the $\frac{3}{2}$-homogeneous solution to the Lam\'e system, which satisfies 
\begin{align*}
p_{\frac{3}{2}}(\hat{e}\cdot x', x_n)\cdot \tilde{e} =0, \text{ for all $\tilde{e}$ such that } \tilde{e}\perp \text{span}\{\hat{e}, e_n\},
\end{align*}
and in polar coordinates $x'\cdot \hat{e}=r\cos(\theta)$, $x_n=r\sin(\theta)$,
\begin{align*}
p_{\frac{3}{2}}(\hat{e}\cdot x', x_n)\cdot e & = r^{\frac{3}{2}}\left(\frac{3\mu+\lambda + \frac{1}{2}(\mu+\lambda)}{4\mu(1+\frac{1}{2})}\cos(\frac{3}{2}\theta)-\frac{\mu+\lambda}{4\mu}\cos(\frac{1}{2}\theta)\right), \\
p_{\frac{3}{2}}(\hat{e}\cdot x', x_n)\cdot e_n &=r^{\frac{3}{2}}\left(\frac{3\mu+\lambda - \frac{1}{2}(\mu+\lambda)}{4\mu(1+\frac{1}{2})}\sin(\frac{3}{2}\theta)-\frac{\mu+\lambda}{4\mu}\sin(\frac{1}{2}\theta)\right).
\end{align*}
\end{defi}

We emphasize that we allow for affine off-sets in our definition of the regular free boundary. Moreover, if $u\in C^{1,\alpha}_{loc}(B_1^+\cup B'_1)$ one has $p_1(x)=-u(x_0)-\nabla u(x_0)\cdot x$.  In \cite{And13} it is shown that when $F=0$, those free boundary points with vanishing order less than $2$ (cf. \eqref{eq:vanishing_u}) are regular. In the next lemma we show that also for a general $F\in L^\infty$ (but for simplicity $\varphi=0$) the vanishing order condition indeed characterizes the regular free boundary points, as in the obstacle problem for $(-\Delta)^{\frac{1}{2}}$.

\begin{lem}
\label{lem:reg_point}
\begin{itemize}
\item [(i)] Let $u\in H^1(B_1^+)$ be a solution to \eqref{eq:Lame1} with $F\in L^\infty(B_1^+)$. Then $x_0 \in \Gamma_u$ is a regular free boundary point if and only if up to an affine function the vanishing order of $u$ at $x_0$ is no larger than $2$, i.e.
\begin{align}\label{eq:vanishing_u}
\limsup_{r\rightarrow 0}\frac{\ln (r^{-\frac{n}{2}}\|u-p_1\|_{L^2(B_r^{+}(x_0))})}{\ln r}<2,
\end{align}
where $p_1(x)=-u(x_0)-\nabla u(x_0)\cdot x$.  
\item [(ii)] Let $\tilde{v}(x')$ be the function from Proposition \ref{prop:Lame_to_obst}, which solves the obstacle problem for $(-\Delta)^{\frac{1}{2}}$ on $\R^{n-1}\times \{0\}$. Then, $x_0\in \Gamma_{\tilde{v}}\cap B_1'$ is a regular free boundary point if and only if $x_0\in \Gamma_u\cap B_1'$ is a regular free boundary point. 
\end{itemize}
\end{lem}

\begin{proof}
Without loss of generality, we assume $x_0=0$ and $u(0)=|Du(0)|=0$. 

\emph{Proof for (i).} One direction is simple. Assume that $\lim_{r\rightarrow 0} \frac{u(rx)}{r^{3/2}}=cp_{3/2}(x_{n-1},x_n)$ for some $c>0$ up to a rotation of coordinates. Then the vanishing order of $u$ at $0$ is equal to $3/2$, thus less than $2$. We will thus focus on the reverse implication.

Assuming \eqref{eq:vanishing_u}, then by \cite[Proposition 7.1]{And13} (which holds for $F\in L^\infty$ as well) there is a sequence $r_j\rightarrow 0$ such that after a rotation of coordinates 
\begin{align}\label{eq:blowup_u}
\frac{u(r_j x)}{r_j^{-\frac{n}{2}}\|u\|_{L^2(B_{r_j}^+)}}\rightarrow c p_{\frac{3}{2}}(x_{n-1},x_n)
\end{align}
in $H^1_{loc}(\R^{n}_+)$ for some $c=c(\mu,\lambda,n)>0$. Let $\bar{u}:=u\eta$ and $\bar u= w+ v$ as in Proposition \ref{prop:Lame_to_obst}, where $w$ solves the auxiliary problem \eqref{eq:Neumann} and $\tilde{v}:=v^n\big|_{\R^{n-1}\times \{0\}}$ solves the obstacle problem for $(-\Delta)^{\frac{1}{2}}$ with the obstacle $\psi:=-\bar{w}$. Since the inhomogeneity satisfies $F\in L^\infty$, analogously as in the proof of Theorem \ref{thm:opt_reg} by the up to the boundary regularity for \eqref{eq:Neumann} and the definition of $\bar{w}$ it follows that $\psi\in C^{1,\gamma}_{loc}(\R^{n-1})$ for any $\gamma \in (0,1)$. Then  \eqref{eq:blowup_u} implies that
the vanishing order of $\tilde{v}-\psi$ at $0$ has to be $\frac{3}{2}$. Indeed, assume this were not true, then by the gap of the vanishing order for the obstacle problem for $(-\Delta)^{\frac{1}{2}}$ with a $C^{1,\gamma}$ obstacle with any $\gamma\in (0,1)$ \cite{RS17}, the vanishing order of $\tilde{v}-\psi$ at $0$ is arbitrarily close to $2$. This together with \eqref{eq:vanishing_u} yields that as $r_j \rightarrow 0$
\begin{align*}
0 \leftarrow \frac{(\tilde{v}-\psi)(r_j x')}{r_j^{-\frac{n}{2}}\|u\|_{L^2(B_{r_j}^+)}} = \frac{u^n(r_j x',0)}{r_j^{-\frac{n}{2}}\|u\|_{L^2(B_{r_j}^+)}}  \text{ in } L^2_{loc}(\R^{n-1}\times \{0\}),
\end{align*}
which is a contradiction to \eqref{eq:blowup_u}. Thus $x_0=0$ is a regular free boundary point for $\tilde{v}$ (by the definition of the regular free boundary point for $(-\Delta)^{\frac{1}{2}}$, see \cite[Definition 4.1 and Proposition 4.2]{KRS14}). This implies 
\begin{align}\label{eq:reg_v}
\lim_{r\rightarrow 0} \frac{u^{n}(rx',0)}{r^{\frac{3}{2}}}=\lim_{r\rightarrow 0} \frac{(\tilde{v}-\psi)(rx')}{r^{\frac{3}{2}}}= \tilde{c} p_{\frac{3}{2}}(x_{n-1},0)\cdot e_n
\end{align}
for some $\tilde{c}>0$ in $H^1_{loc}(\R^n_+)$. Moreover, the free boundary $\Gamma_{\tilde{v}}$ (and thus $\Gamma_u$) is a $C^{1,\beta}$-graph in a neighborhood of $0$. This together with the optimal $C^{1,\frac{1}{2}}_{loc}$ regularity of the solution implies that along subsequences the limits of $u(rx)/r^{3/2}$ as $r\rightarrow 0$ exist and solve the global Lam\'e system with the flat free boundary $\{x_{n-1}=x_n=0\}$. Since the restriction of the $n$th-component of the limits to $\R^{n-1}\times \{0\}$ is unique and equal to $\tilde{c}p_{3/2}(x_{n-1},0)\cdot e_n$, we can conclude that the limit $\lim_{r\rightarrow 0}u(rx)/r^{3/2}$ exists and is equal to $\tilde{c}p_{3/2}(x_{n-1},x_n)$.

\emph{Proof for (ii).} First by Proposition \ref{prop:Lame_to_obst} one has that $\Gamma_u\cap K =\Gamma_{\tilde v}\cap K$ for any $K\Subset B'_1$. On the hand, let $x_0\in \Gamma_{\tilde{v}}\cap B'_1$ be a regular free boundary point, then \eqref{eq:reg_v} holds true at $x_0$ up to a rotation of coordinates. A similar argument as in the proof for (i) yields that $\lim_{r\rightarrow 0} u(rx)/r^{3/2}$ exists and is equal to $\tilde{c}p_{3/2}$ for some $\tilde c>0$. This means that  $x_0$ is also a regular free boundary point for the Lam\'e system on $\Gamma_u$.   On the other hand, given $x_0$ a regular free boundary point on $\Gamma_u \cap B_1'$, by the definition of $\tilde{v}$ and Definition \ref{def:reg_fbp}  it holds that $x_0 \in \Gamma_{\tilde{v}}\cap B_1'$ is also a regular free boundary point for the obstacle problem for the half-Laplacian. 
\end{proof}

Using our above established relation between the problems for the Lam\'e free boundary problem and the half-Laplacian (see Proposition \ref{prop:Lame_to_obst}), the regularity and the higher regularity of the free boundary of the Lam\'e problem follows immediately from that for the fractional Laplacian (see \cite{AR20} and \cite{KRSIV}) by carrying out a truncation strategy as in \cite{KRSIV}.

\begin{proof}[Proof of Theorem \ref{thm:higher_reg}]
We first discuss (i): For simplicity of notation we assume that $x_0=0$. Seeking to obtain a local result, we consider $\bar{u}:= u \eta$ for $\eta$ as in Proposition \ref{prop:Lame_to_obst} and deduce estimates in any set $B_{\tilde{r}}^+ \subset B_r^+$ where $\eta = 1$. In the set $B_r^+$, the local regularity of the auxiliary function $w$ from Step (ii) in the proof of Proposition \ref{prop:Lame_to_obst} is determined by $F$ only. In particular, since the inhomogeneity satisfies $F\in C^{\theta-\frac{3}{2}}(B_r^+\cup B'_r)$, then by the up to the boundary Schauder estimate for Neumann problem for the Lam\'e system (cf. \cite{ADNII}), $w$ is $C^{\theta+\frac{1}{2}}(B_{\tilde{r}}^+\cup B_{\tilde{r}}')$. Additionally, Proposition \ref{prop:more_gen_prop_reg} yields integrability, i.e., $w\in C^{\theta+\frac{1}{2}}_{loc}(B_{\tilde{r}}^+\cup B_{\tilde{r}}')\cap \dot{H}^2(\R^n_+)$. Therefore, by Proposition \ref{prop:Lame_to_obst}, we infer $\tilde{h}=h+2\mu\p_n w^n+\lambda \di w\in C^{\theta-\frac{1}{2}}_{loc}( B_{\tilde{r}}')\cap L^2(\R^{n-1})$ and, as a consequence of Proposition \ref{prop:Lame_to_obst}, the interaction of pseudodifferential operators and Hölder-Zygmund spaces (see for instance \cite[Chapter VI, Section 5.3]{Stein}), and of the pseudolocality of the fractional Laplacian (as a pseudodifferential operator, see \cite[Chapter VI, Section 2]{Stein}), $\bar{w}=(-\Delta_{x'})^{-\frac{1}{2}}\tilde h\in C^{\theta + \frac{1}{2}}_{loc}(B_{\tilde{r}}')$. Hence $\tilde{v}(x'): =u^n(x',0)\eta(x',0)-\bar{w}(x')$ solves the obstacle problem for $(-\Delta_{x'})^{\frac{1}{2}}$ with the obstacle $\psi=\phi\eta-\bar{w} \in C^{\theta+\frac{1}{2}}_{loc}(B_{\tilde{r}}')$. 
By Lemma \ref{lem:reg_point} a regular point of the Lam\'e system is mapped to a regular point for the half-Laplacian and vice-versa.
By \cite[Theorem 1.2]{AR20} or \cite[Theorem 2]{KRSIV}, around the regular point $0\in \Gamma_{\tilde{v}}\cap B'_1$ the free boundary is $C^\theta$ regular. Since restricted to the support of $\eta$
the free boundary satisfies $\Gamma_{u}=\Gamma_{\tilde{v}}$, we have that $\Gamma_u$ is also $C^\theta$ regular in a neighborhood of $x_0=0$.

Finally, (ii) follows analogously by the results of \cite{KRSIV} for the analyticity of the regular free boundary for the obstacle problem for half-Laplacian with real analytic obstacle, once the (tangential) analyticity of $w|_{B_r'(x_0)}$, $\di(w)|_{B_r'(x_0)}$ and $\p_n w|_{B_r'(x_0)}$ (and thus of $\psi$) is established for some $r>0$. It hence remains to discuss the analyticity of these functions: This however is a consequence the fact that the data in the problem for $w$ (cf. \eqref{eq:Neumann}) are analytic in $B_r^+(x_0)\cup B'_r(x_0)$ (with $\eta$ chosen to be one in $B_r(x_0)$, $r>0$ small). By Proposition \ref{prop:more_gen_prop_reg} we obtain the analyticity of $w$ and thus of $\overline{w}$ and $\psi$  in a neighborhood of $x_0$.
\end{proof}

Last but not least, we turn to the proof of Theorem \ref{thm:Lame_sing_set}.

\begin{proof}[Proof of Theorem \ref{thm:Lame_sing_set}]
It suffices to establish that if $x_0 \in \Sigma_u \cap B_1'$, then $x_0 \in \Sigma_{\tilde{v}}\cap B_1'$. The remaining arguments then follow from the properties of the singular set of the Signorini problem for the Laplace operator (see (iii) at the beginning of this section) noting that the corresponding obstacle $\psi=-\bar w$ is smooth in a neighborhood of $x_0$. Since $\Lambda_{\tilde{v}}\cap K = \Lambda_{u}\cap K$ for any $K\Subset B_1'$ the implication that if $x_0 \in \Sigma_u \cap B_1'$, then $x_0 \in \Sigma_{\tilde{v}}\cap B_1'$ is immediate.
\end{proof}

\begin{rmk}
We remark that extensions of Theorem \ref{thm:Lame_sing_set} to settings with non-zero obstacle are possible but require a more careful investigation of the vanishing order of the blow-ups. 
\end{rmk}

\section*{Acknowledgements}
A.R. was supported by the Deutsche Forschungsgemeinschaft (DFG, German ResearchFoundation) under Germany’s Excellence Strategy EXC-2181/1 - 390900948 (the Heidelberg STRUCTURES Cluster of Excellence).

\bibliography{lame_neu}

\begin{thebibliography}{10}

\bibitem{AR20}
N.~Abatangelo and X.~Ros-Oton.
\newblock Obstacle problems for integro-differential operators: higher
  regularity of free boundaries.
\newblock {\em Adv. Math.}, 360:106931, 61, 2020.

\bibitem{ADNII}
S.~Agmon, A.~Douglis, and L.~Nirenberg.
\newblock Estimates near the boundary for solutions of elliptic partial
  differential equations satisfying general boundary conditions {I}{I}.
\newblock {\em Communications on Pure and Applied Mathematics}, 17(1):35--92,
  1964.

\bibitem{And13}
J.~Andersson.
\newblock Optimal regularity for the {S}ignorini problem and its free boundary.
\newblock {\em Invent. Math.}, 204(1):1--82, 2016.

\bibitem{AC04}
I.~Athanasopoulos and L.~A. Caffarelli.
\newblock Optimal regularity of lower dimensional obstacle problems.
\newblock {\em Zap. Nauchn. Sem. S.-Peterburg. Otdel. Mat. Inst. Steklov.
  (POMI)}, 310(Kraev. Zadachi Mat. Fiz. i Smezh. Vopr. Teor. Funkts. 35
  [34]):49--66, 226, 2004.

\bibitem{ACS08}
I.~Athanasopoulos, L.~A. Caffarelli, and S.~Salsa.
\newblock The structure of the free boundary for lower dimensional obstacle
  problems.
\newblock {\em Amer. J. Math.}, 130(2):485--498, 2008.

\bibitem{CRS17}
L.~Caffarelli, X.~Ros-Oton, and J.~Serra.
\newblock Obstacle problems for integro-differential operators: regularity of
  solutions and free boundaries.
\newblock {\em Invent. Math.}, 208(3):1155--1211, 2017.

\bibitem{CSS}
L.~Caffarelli, S.~Salsa, and L.~Silvestre.
\newblock Regularity estimates for the solution and the free boundary of the
  obstacle problem for the fractional {L}aplacian.
\newblock {\em Invent. Math.}, 171(2):425--461, 2008.

\bibitem{CS07}
L.~Caffarelli and L.~Silvestre.
\newblock An extension problem related to the fractional {L}aplacian.
\newblock {\em Communications in partial differential equations},
  32(8):1245--1260, 2007.

\bibitem{CSV20}
M.~Colombo, L.~Spolaor, and B.~Velichkov.
\newblock Direct epiperimetric inequalities for the thin obstacle problem and
  applications.
\newblock {\em Communications on Pure and Applied Mathematics}, 73(2):384--420,
  2020.

\bibitem{DS17}
D.~Danielli and S.~Salsa.
\newblock Obstacle problems involving the fractional {L}aplacian.
\newblock In {\em Recent Developments in Nonlocal Theory}, pages 81--164.
  Sciendo, 2017.

\bibitem{DSS16}
D.~de~Silva and O.~Savin.
\newblock Boundary {H}arnack estimates in slit domains and applications to thin
  free boundary problems.
\newblock {\em Revista matem{\'a}tica iberoamericana}, 32(3):891--912, 2016.

\bibitem{DL76}
G.~Duvaut and J.-L. Lions.
\newblock {\em Inequalities in mechanics and physics}.
\newblock Springer-Verlag, Berlin, 1976.

\bibitem{FR20}
X.~Fern{\'a}ndez-Real.
\newblock The thin obstacle problem: A survey.
\newblock {\em arXiv preprint arXiv:2011.03299}, 2020.

\bibitem{Fi63}
G.~Fichera.
\newblock Sul problema elastostatico di {S}ignorini con ambigue condizioni al
  contorno.
\newblock {\em Atti Accad. Naz. Lincei Rend. Cl. Sci. Fis. Mat. Nat. (8)},
  34:138--142, 1963.

\bibitem{FS16}
M.~Focardi and E.~Spadaro.
\newblock An epiperimetric inequality for the thin obstacle problem.
\newblock {\em Advances in Differential Equations}, 21(1/2):153--200, 2016.

\bibitem{GP09}
N.~Garofalo and A.~Petrosyan.
\newblock Some new monotonicity formulas and the singular set in the lower
  dimensional obstacle problem.
\newblock {\em Invent. Math.}, 177(2):415--461, 2009.

\bibitem{GPSVG15}
N.~Garofalo, A.~Petrosyan, and M.~Smit Vega~Garcia.
\newblock An epiperimetric inequality approach to the regularity of the free
  boundary in the {S}ignorini problem with variable coefficients.
\newblock {\em J. Math. Pures Appl. (9)}, 105(6):745--787, 2016.

\bibitem{GSVG14}
N.~Garofalo and M.~Smit Vega~Garcia.
\newblock New monotonicity formulas and the optimal regularity in the
  {S}ignorini problem with variable coefficients.
\newblock {\em Advances in Mathematics}, 262:682--750, 2014.

\bibitem{Gu}
N.~Guillen.
\newblock Optimal regularity for the {S}ignorini problem.
\newblock {\em Calculus of Variations and Partial Differential Equations},
  36(4):533--546, 2009.

\bibitem{HoermanderI}
L.~H{\"o}rmander.
\newblock {\em The analysis of linear partial differential operators {I}:
  {D}istribution theory and {F}ourier analysis}.
\newblock Springer, 2015.

\bibitem{JN17}
Y.~Jhaveri and R.~Neumayer.
\newblock Higher regularity of the free boundary in the obstacle problem for
  the fractional {L}aplacian.
\newblock {\em Adv. Math.}, 311:748--795, 2017.

\bibitem{K96}
K.~Kato.
\newblock New idea for proof of analyticity of solutions to analytic nonlinear
  elliptic equations.
\newblock {\em SUT J. Math}, 32(2):157--161, 1996.

\bibitem{KL12}
H.~Koch and T.~Lamm.
\newblock Geometric flows with rough initial data.
\newblock {\em Asian Journal of Mathematics}, 16(2):209--235, 2012.

\bibitem{KPS}
H.~Koch, A.~Petrosyan, and W.~Shi.
\newblock Higher regularity of the free boundary in the elliptic {S}ignorini
  problem.
\newblock {\em Nonlinear Anal.}, 126:3--44, 2015.

\bibitem{KRS14}
H.~Koch, A.~R{\"u}land, and W.~Shi.
\newblock The variable coefficient thin obstacle problem: Carleman estimates.
\newblock {\em Advances in Mathematics}, 301(1):820 -- 866, 2016.

\bibitem{KRSIV}
H.~Koch, A.~R\"{u}land, and W.~Shi.
\newblock The variable coefficient thin obstacle problem: higher regularity.
\newblock {\em Adv. Differential Equations}, 22(11-12):793--866, 2017.

\bibitem{KRSI}
H.~Koch, A.~R{\"u}land, and W.~Shi.
\newblock The variable coefficient thin obstacle problem: Optimal regularity,
  free boundary regularity and first order asymptotics.
\newblock {\em Annales de l'Institut Henri Poincare (C) Non Linear Analysis},
  34(4):845 -- 897, 2017.

\bibitem{KRS16}
H.~Koch, A.~R{\"u}land, and W.~Shi.
\newblock Higher regularity for the fractional thin obstacle problem.
\newblock {\em New York J. Math.}, 25:745--838, 2019.

\bibitem{Mitrea13}
D.~Mitrea.
\newblock {\em Distributions, partial differential equations, and harmonic
  analysis}.
\newblock Universitext. Springer, New York, 2013.

\bibitem{MN57}
C.~B. Morrey~Jr and L.~Nirenberg.
\newblock On the analyticity of the solutions of linear elliptic systems of
  partial differential equations.
\newblock {\em Communications on Pure and Applied Mathematics}, 10(2):271--290,
  1957.

\bibitem{PSU}
A.~Petrosyan, H.~Shahgholian, and N.~N. Uraltseva.
\newblock {\em Regularity of free boundaries in obstacle-type problems}, volume
  136.
\newblock American Mathematical Soc., 2012.

\bibitem{RD20}
X.~Ros-Oton and T.-L. Dami\`{a}.
\newblock New boundary {H}arnack inequalities with right hand side.
\newblock {\em arXiv:2010.01064}, 2020.

\bibitem{ROS17}
X.~Ros-Oton and J.~Serra.
\newblock The structure of the free boundary in the fully nonlinear thin
  obstacle problem.
\newblock {\em Adv. Math.}, 316:710--747, 2017.

\bibitem{RS17}
A.~R\"{u}land and W.~Shi.
\newblock Optimal regularity for the thin obstacle problem with
  {$C^{0,\alpha}$} coefficients.
\newblock {\em Calc. Var. Partial Differential Equations}, 56(5):Paper No. 129,
  41, 2017.

\bibitem{Schumann1}
R.~Schumann.
\newblock Regularity for {S}ignorini's problem in linear elasticity.
\newblock {\em Manuscripta Math.}, 63(3):255--291, 1989.

\bibitem{Schumann2}
R.~Schumann.
\newblock A remark on a boundary contact problem in linear elasticity.
\newblock {\em Manuscripta Math.}, 63(4):455--468, 1989.

\bibitem{S33}
A.~Signorini.
\newblock Sopra alcune questioni di statica dei sistemi continui.
\newblock {\em Annali della Scuola Normale Superiore di Pisa-Classe di
  Scienze}, 2(2):231--251, 1933.

\bibitem{S59}
A.~Signorini.
\newblock Questioni di elasticit{\`a} non linearizzata e semilinearizzata.
\newblock {\em Rend. Mat. Appl}, 18(5):95--139, 1959.

\bibitem{S07}
L.~Silvestre.
\newblock Regularity of the obstacle problem for a fractional power of the
  {L}aplace operator.
\newblock {\em Communications on Pure and Applied Mathematics: A Journal Issued
  by the Courant Institute of Mathematical Sciences}, 60(1):67--112, 2007.

\bibitem{Stein}
E.~M. Stein and T.~S. Murphy.
\newblock {\em Harmonic analysis: real-variable methods, orthogonality, and
  oscillatory integrals}, volume~3.
\newblock Princeton University Press, 1993.

\bibitem{T07}
M.~E. Taylor.
\newblock {\em Tools for {P}{D}{E}: pseudodifferential operators,
  paradifferential operators, and layer potentials}.
\newblock Number~81. American Mathematical Soc., 2007.

\end{thebibliography}
\bibliographystyle{abbrv}

\end{document}